\documentclass[12pt,a4paper]{amsart}
\usepackage[english]{babel}
\usepackage{amssymb,latexsym,amsfonts,amsthm,upref,amsmath}
\usepackage[margin=1in]{geometry}
\usepackage[colorlinks=true]{hyperref}
\usepackage{xcolor}
\hypersetup{citecolor=blue, urlcolor=red, colorlinks=false, linkcolor=green, citecolor=blue, filecolor=magenta}

\newtheorem{theorem}{Theorem}[section]
\newtheorem{lemma}{Lemma}[subsection]
\newtheorem{proposition}{Proposition}[subsection]

\newtheorem{example}{Example}[subsection]

\numberwithin{equation}{section}

\begin{document}

\title{Quasi-particle Bases of Principal Subspaces of the Affine Lie Algebra of Type \texorpdfstring{$G_\MakeLowercase{2}^{(1)}$}{G_2^{(1)}} }

\author{Marijana Butorac}

\address{University of Rijeka, Department of Mathematics, Radmile Matej\v{c}i\'{c} 2, 51000 Rijeka, Croatia}

\email{mbutorac@math.uniri.hr}


\thanks{This work has been supported in part by the Croatian Science Foundation under the project 2634., by the Croatian Scientific Centre of Excellence Quantix Lie and by University of Rijeka research grant 13.14.1.2.02.}

\begin{abstract} 
The aim of this work is to construct the quasi-particle basis of principal subspace of standard module of highest weight $k\Lambda_0$ of level $k\geq 1$ of affine Lie algebra of type $G_2^{(1)}$ by means of which we obtain the basis of principal subspace of generalized Verma module.\end{abstract}

\maketitle

\section*{INTRODUCTION}
Principal subspaces of standard modules of affine Lie algebras $A_1^{(1)}$ were first introduced by B. L. Feigin and A. V. Stoyanovsky in \cite{FS}. Motivated by the work of J. Lepowsky and M. Primc (\cite{LP}), Feigin and Stoyanovsky related characters of principal subspaces with Rogers-Ramanujan type identities. This connection was further studied by many authors, in particular in \cite{AKS}, \cite{Cal1}--\cite{Cal2}, \cite{CalLM1}--\cite{CalLM4}, \cite{CalMP}, \cite{CLM1}--\cite{CLM2}, \cite{G}, \cite{KP}, \cite{PS1}--\cite{PS2},  \cite{S1}--\cite{S2} and others. More recently, Slaven Ko\v zi\' c in \cite{Ko1}--\cite{Ko2} showed that cha\-racter formulas for level $1$ principal subspaces associated with the integrable highest weight module of quantum affine algebra $U_q(\widehat{\mathfrak{sl}_{2}})$ coincide with
the character formulas found in \cite{FS}. 

In \cite{G}, G. Georgiev constructed bases for principal subspaces of certain standard $A_l^{(1)}$-modules by using monomials of certain vertex operator coefficients corresponding to simple roots of $A_l$, the so-called quasi-particles (cf. \cite{FS}), from which were easily obtained the Rogers-Ramanujan type character formulas. In \cite{Bu1} and \cite{Bu2} we extended Georgiev's construction of quasi-particle bases for principal subspaces of standard module $L(k\Lambda_0)$ and generalized Verma module $N(k\Lambda_0)$ of highest weight $k\Lambda_0$, $k \in \mathbb{N}$ for affine Lie algebras of type $B_l^{(1)}$ and $C_l^{(1)}$, $l \geq 2$. As a consequence we proved two new series of Rogers-Ramanujan type identities obtained from the characters of principal subspaces of generalized Verma module.

In this note we construct quasi-particle bases of principal subspaces of generalized Verma module $N(k\Lambda_0)$ and its irreducible quotient in the case of affine Lie algebra of type $G_2^{(1)}$. Two main steps in the construction are similar to the case of $B_2^{(1)}$. First step is to find relations among quasi-particles from which follow the spanning set of principal subspaces and the second step is to prove that the spanning set is linearly independent by induction on the linear order on quasi-particles. The main differences with the case of $B_2^{(1)}$ are relations which describe the interaction of quasi-particles associated to different simple roots and operators which we use in the proof of linear independence, since we don't have a simple current operator as in the proof of independence for $B_2^{(1)}$. 

To state our main results, denote by $W_{L(k\Lambda_0)}$ the principal subspace of level $k$ standard module and by $\text{ch} \ W_{L(k\Lambda_{0})}$ the character of $W_{L(k\Lambda_0)}$ and by $W_{N(k\Lambda_0)}$ the principal subspace of generalized Verma module $N(k\Lambda_0)$. Our result states:
\begin{theorem}
\begin{equation*}
\mathrm{ch} \  W_{L(k\Lambda_{0})}
= \sum_{\substack{r^{(1)}_{1}\geq \ldots \geq r^{(k)}_{1}\geq 0\\ r^{(1)}_{2}\geq \ldots \geq r^{(3k)}_{2}\geq 0}}
\frac{q^{\sum_{s=1}^k r^{(s)^{2}}_{1}+\sum_{s=1}^{3k}r^{(s)^{2}}_{2}-\sum_{s=1}^k r^{(s)}_{1}(r^{(3s)}_{2}+r^{(3s-1)}_{2}+r^{(3s-2)}_{2})}}{(q)_{r^{(1)}_{1}-r^{(2)}_{1}}\cdots (q)_{r^{(k)}_{1}}(q)_{r^{(1)}_{2}-r^{(2)}_{2}}\cdots (q)_{r^{(3k)}_{2}}}y^{r_1}_{1}y^{r_2}_{2},
\end{equation*}
where $r_1=\sum_{s=1}^{k}r_1^{(s)}$ and $r_2=\sum_{s=1}^{3k}r_2^{(s)}$.
\end{theorem}
This new fermionic formula which follows directly from quasi-particle basis of $W_{L(k\Lambda_{0})}$ is related to the study of parafermionic Rogers-Ramanujan type characters \cite{Ge}.

We use quasi-particle bases of $W_{L(k\Lambda_{0})}$ in the construction of quasi-partic\-les bases of principal subspace $W_{N(k\Lambda_{0})}$ of generalized Verma module, from which follows a generalization of Euler-Cauchy identity  
\begin{theorem}\label{tm1}
$$\prod_{m > 0} \frac{1}{(1-q^my_1)}\frac{1}{(1-q^my_2)}\frac{1}{(1-q^my_1y_2)}\frac{1}{(1-q^my_1y_2^2)}\frac{1}{(1-q^my_1y_2^3)}\frac{1}{(1-q^my_1^2y_2^3)}$$
\begin{equation}\label{eq:1}= \sum_{\substack{r^{(1)}_{1}\geq r^{(2)}_{1}\geq r^{(3)}_{1} \geq \ldots  \geq  0\\ r^{(1)}_{2}\geq r^{(2)}_{2} \geq r^{(3)}_{2}\geq \ldots  \geq 0}}
\frac{q^{\sum_{s\geq 1}  r^{(s)^{2}}_{1}+\sum_{s\geq 1} r^{(s)^{2}}_{2}-\sum_{s\geq 1} r^{(s)}_{1}(r^{(3s)}_{2}+r^{(3s-1)}_{2}+r^{(3s-2)}_{2})}}{(q)_{r^{(1)}_{1}-r^{(2)}_{1}}(q)_{r^{(2)}_{1}-r^{(3)}_{1}}\cdots  (q)_{r^{(1)}_{2}-r^{(2)}_{2}}(q)_{r^{(2)}_{2}-r^{(3)}_{2}}\cdots }y^{r_1}_{1}y^{r_2}_{2},
\end{equation}
where $r_1=\sum_{s\geq 1} r_1^{(s)}$ and $r_2=\sum_{s\geq 1} r_2^{(s)}$.
The sum on the right side of (\ref{eq:1}) is over all descending infinite sequences of non-negative integers with finite support. \end{theorem}

\section{Principal subspaces}
\label{sec:1}
Let $\mathfrak{g}$ be a complex simple Lie algebra of type $G_2$ with a triangular decompositi\-on $\mathfrak{g} =\mathfrak{n}_{-}\oplus \mathfrak{h}\oplus \mathfrak{n}_{+}$, with the basis 
$\Pi=\{\alpha_1=\frac{1}{\sqrt{3}}(-2\epsilon_1+\epsilon_2+\epsilon_3) , \alpha_2={\frac{1}{\sqrt{3}}(\epsilon_1-\epsilon_2)} \}$ of the root system $R$ and the co\-rresponding set of fundamental weights $\left\{\omega_1=2\alpha_1+3\alpha_2\right.$, $\left.\omega_2=\alpha_1+2\alpha_2\right\}$, where $\epsilon_1,\epsilon_2,\epsilon_3$ are vectors of the standard basis of $\mathbb{R}^3$. Denote by $\theta=\frac{1}{\sqrt{3}}(-\epsilon_1-\epsilon_2+2\epsilon_3) $ the highest root and assume that all long roots $\alpha \in R$ are normalized by the condition $\langle \alpha,\alpha \rangle =2$, where $\langle \cdot,\cdot \rangle$ denotes the invariant nondegenarate bilinear form on $\mathfrak{g}$, which induces a bilinear form on $\mathfrak{h}^*$. Denote by $Q$ the root lattice and by $P$ the weight lattice of $\mathfrak{g}$. Then, $P=Q$. For later use we fix root vectors  
\begin{align}\label{eq:2}
x_{\alpha_1+\alpha_2}=\left[x_{\alpha_2},x_{\alpha_1} \right], \ x_{\alpha_1+2 \alpha_2}=\left[x_{\alpha_2},x_{\alpha_1+\alpha_2} \right],\\
\nonumber
x_{\alpha_1+3 \alpha_2}=\left[x_{\alpha_2},x_{\alpha_1+2\alpha_2} \right], \ x_{2\alpha_1+3\alpha_2}=\left[x_{\alpha_1},x_{\alpha_1+3\alpha_2} \right].
\end{align}

Let $\widetilde{\mathfrak{g}}$ be the associated affine Lie algebra
$$\widetilde{ \mathfrak{g} }=\widehat{\mathfrak{g}} \oplus \mathbb{C}d,
$$
$$\widehat{\mathfrak{g}}= \mathfrak{g}\otimes \mathbb{C}\left[t,t^{-1}\right] \oplus \mathbb{C}c,
$$
with commutation relations
\begin{equation}\label{eq:3}
\left[x(j_1),y(j_2)\right]= \left[x, y\right](j_1+j_2) + \left\langle x, y \right\rangle j_1 \delta_{j_1+j_2,0}c,
\end{equation}
$$\left[c,\widetilde{ \mathfrak{g} }\right]= 0, \ \ \left[d,x(j)\right]= jx(j),
$$
where $x(j) = x \otimes t^j$ for $x, y \in \mathfrak{g}, \ j,j_1,j_2 \in \mathbb{Z}$, (cf. \cite{K}). 
We consider $\widetilde{\mathfrak{g}}$-subalgebras
$$ \mathcal{L}(\mathfrak{n}_{+})=\mathfrak{n}_{+} \otimes \mathbb{C}[t,t^{-1}],
$$
\begin{align*}    
\mathcal{L}(\mathfrak{n}_{+})_{\geq 0}= \mathfrak{n}_{+}  \otimes \mathbb{C}[t], \ \ \mathcal{L}(\mathfrak{n}_{+})_{< 0}= \mathfrak{n}_{+}  \otimes t^{-1}\mathbb{C}[t^{-1}]
\end{align*}
and
\begin{equation*}
 \mathcal{L}(\mathfrak{n}_{\alpha})=\mathfrak{n}_{\alpha} \otimes \mathbb{C}[t,t^{-1}],
\end{equation*}
where 
$$ \mathfrak{n}_{\alpha}=\mathbb{C}x_{\alpha}
$$
are one-dimensional $\mathfrak{g}$-subalgebras generated with root vectors $x_{\alpha}$, $\alpha \in R$. 

We extend our form $\langle \cdot,\cdot \rangle$ to $\widetilde{\mathfrak{h}}= \mathfrak{h} \oplus  \mathbb{C}c \oplus \mathbb{C}d$. The set of simple roots of $\widetilde{\mathfrak{g}}$ is $\left\{\alpha_0, \alpha_1, \alpha_2\right\}$ and $\left\{\Lambda_0, \Lambda_1, \Lambda_2\right\}$ is the set of fundamental weights. Denote by $L(\Lambda_0)$ a standard (i.e. integrable highest weight) $\widetilde{\mathfrak{g}}$-module of level 1 with the highest weight vector $v_{L(\Lambda_0)}$.

Fix $k \in \mathbb{N}$. Denote by $N(k\Lambda_0)$ the generalized Verma module and by $L(k\Lambda_0)$ its irreducible quotient. The induced $\widetilde{\mathfrak{g}}$-module $N(k\Lambda_0)$ is defined as
$$
N(k\Lambda_{0})= 
U(\widehat{\mathfrak{g}})\otimes_{U(\widehat{\mathfrak{g}}_{\geq 0})} \mathbb{C}v_{N(k\Lambda_{0})},
$$
where $\widehat{\mathfrak{g}}_{\geq 0}=\bigoplus_{n\geq0} \mathfrak{g}\otimes t^{n}\oplus \mathbb{C}c$ 
and $\mathbb{C}v_{N(k\Lambda_{0})}$ is 1-dimensional $\widehat{\mathfrak{g}}_{\geq 0}$-module, such that 
$$
cv_{N(k\Lambda_{0})}=kv_{N(k\Lambda_{0})}, \ \ dv_{N(k\Lambda_{0})}=0, \ \ (\mathfrak{g}\otimes t^{j})v_{N(k\Lambda_{0})}=0, \ \ j \geq 0.
$$
Set 
$$
v_{N(k\Lambda_{0})}=1 \otimes v_{N(k\Lambda_{0})}.
$$

The generalized Verma module has a structure of a vertex operator algebra, as its irreducible quotient $L(k\Lambda_0)$ and all the level k standard modules are modules for vertex operator algebra $L(k\Lambda_0)$. The vertex operator map is determined by
$$
Y(x(-1)v_{N(k\Lambda_{0})}, z)=\sum_{m \in \mathbb{Z}}x(m)z^{-m-1}=x(z)
$$
for $x \in \mathfrak{g}$ (cf. \cite{LL}).
We will use the commutator formula among vertex operators: 
\begin{align}\label{eq:4} 
[Y(x_{\alpha}(-1)v_{N(k\Lambda_{0})},z_1), Y(x_{\beta}(-1)^rv_{N(k\Lambda_{0})},z_2)]\\
\nonumber
= \sum_{j \geq 0} \frac{(-1)^j}{j!} \left(\frac{d}{dz_1} 
 \right)^j z^{-1}_2
\delta\left(\frac{z_1}{z_2}\right)Y(x_{\alpha}(j)x_{\beta}(-1)^rv_{N(k\Lambda_{0})},z_2),
\end{align}
where $\alpha, \beta \in R$, (cf. \cite{FHL}).

Denote by $v_{L(k\Lambda_{0})}$ the highest weight vector of $L(k\Lambda_0)$. We define a principal subspace $W_{L(k\Lambda_0)}$ of $L(k\Lambda_0)$ (see \cite{FS}, \cite{G}) as
$$
W_{L(k\Lambda_0)}= U(\mathcal{L}(\mathfrak{n}_+))v_{L(k\Lambda_0)}
$$
and the principal subspace $W_{N(k\Lambda_{0})}$ of the generalized Verma module $N(k\Lambda_{0})$ as
$$
W_{N(k\Lambda_{0})}= U(\mathcal{L}(\mathfrak{n}_{+}))v_{N(k\Lambda_{0})}.
$$

Note that the map 
\begin{align*}    
f: U(\mathcal{L}(\mathfrak{n}_{+})_{< 0}) \rightarrow W_{N(k\Lambda_{0})},\\
f(b)=bv_{N(k\Lambda_{0})}
\end{align*}    
is an isomorphism of $\mathcal{L}(\mathfrak{n}_{+})_{< 0}$-modules. If we order basis elements of $\mathfrak{n}_{+}$
$$
\left\{x_{\alpha_1}, x_{\alpha_2}, x_{\alpha_1+\alpha_2},x_{\alpha_1+2\alpha_2},x_{\alpha_1+3\alpha_2},x_{2\alpha_1+3\alpha_2}\right\}
$$  
in the following way:
$$
x_{\alpha_2} < x_{\alpha_1} <x_{\alpha_1+\alpha_2}< x_{\alpha_1+2\alpha_2}<x_{\alpha_1+3\alpha_2}<x_{2\alpha_1+3\alpha_2}
$$
and basis elements of $\mathcal{L}(\mathfrak{n}_{+})_{< 0}$ 
$$\left\{x_{\alpha}(m): \alpha \in R_+,  m<0\right\}$$ 
as:
$$
x(m) \leq y(m') \ \ \Leftrightarrow \ \ x<y \ \ \text{or} \ \ x=y \ \ \text{and} \ \ m <m',
$$
then from the Poincar\'{e}-Birkhoff-Witt theorem follows that 
vectors 
\begin{equation}\label{eq:5}
x_{\alpha_2}(m^1_1)\cdots  x_{\alpha_2}(m^{s_1}_1)x_{\alpha_1}(m^1_2)\cdots x_{\alpha_1}(m^{s_2}_2)
\end{equation}
$$ \cdots x_{2\alpha_1+3\alpha_2}(m^{1}_6)\cdots   x_{2\alpha_1+3\alpha_2}(m^{s_{6}}_{6})v_{N(k\Lambda_{0})},$$
where $m_i^1\leq \cdots \leq m_i^{s_i}<0$, $s_i \geq 0$, $1 \leq i \leq 6$, form a basis of a vector space ${W_{N(k\Lambda_{0})}}$. 

In next sections, we construct bases of principal subspaces $W_{L(k\Lambda_{0})}$ and $W_{N(k\Lambda_{0})}$ in terms of certain coefficients of vertex operators corresponding to vectors $x_{\alpha_i}(-1)^rv_{L(k\Lambda_0)}$ (and $x_{\alpha_i}(-1)^rv_{N(k\Lambda_0)}$), where $r\geq 1$ and $\alpha_i \in \Pi$. 

First, we choose a special subspace of $U(\mathcal{L}(\mathfrak{n}_+))$
$$
U=U(\mathcal{L}(\mathfrak{n}_{\alpha_2}))U(\mathcal{L}(\mathfrak{n}_{\alpha_1})).
$$
It is easy to see that principal subspaces are generated by operators in $U$ acting on the highest weight vectors $v_{L(k\Lambda_0)}$ and $v_{N(k\Lambda_0)}$ (see Lemma 3.1 in \cite{G}).

\section{Quasi-particle bases of principal subspaces}
\label{sec:2}
We start this section with introducing all necessary notions and facts needed in the construction of quasi-particle bases of principal subspaces. Some terms and labels which we use, but are not mentioned, are the same as in our previous work, therefore, for more details we refer to \cite{Bu1}--\cite{Bu2} and also to \cite{G}. 

\subsection{Quasi-particle monomials}
\label{sec:3}
For given $i \in \left\{1,2\right\}$, $r \in \mathbb{N}$ and $m\in \mathbb{Z}$ define a quasi-particle 
of color $i$, charge $r$ and energy $-m$ by
\begin{equation*} 
x_{r\alpha_{i}}(m)=\textup{Res}_z \left\{ z^{m+r-1}x_{r\alpha_{i}}(z)\right\},
\end{equation*}
where $x_{r\alpha_{i}}(z)$ is a vertex operator
$$
 x_{r\alpha_{i}}(z):= x_{\alpha_i}(z)^r=Y(\left(x_{\alpha_i}(-1)\right)^rv_{L(k\Lambda_{0})},z).
$$
$x_{r\alpha_{i}}(z)$ is the generating function of quasi-particles of color $i$ and charge $r$. 

Denote by $b(\alpha_i)$ the monochromatic quasi-particle monomial, that is the product of quasi-particles of the same color $i$. We say that monomial $b$ ``colored'' with more co\-lors is a polychromatic monomial. As in the case of $B_2^{(1)}$, our basis monomials will be ``colored'' with two colors $i=1,2$ and our monomials will have the form
$$
b= b(\alpha_{2})b(\alpha_{1}).
$$
For monomial
$$
b(\alpha_{2})b(\alpha_{1})=x_{n_{r_{2}^{(1)},2}\alpha_{2}}(m_{r_{2}^{(1)},2}) \cdots  x_{n_{1,2}\alpha_{2}}(m_{1,2}) 
x_{n_{r_{1}^{(1)},1}\alpha_{1}}(m_{r_{1}^{(1)},1}) \cdots  x_{n_{1,1}\alpha_{1}}(m_{1,1}),
$$ 
we will say it is of charge-type 
$$
\mathcal{R}'=\left(n_{r_{2}^{(1)},2}, \ldots ,n_{1,2};n_{r_{1}^{(1)},1}, \ldots ,n_{1,1}\right),
$$
where
$$
0 \leq n_{r_{i}^{(1)},i}\leq \ldots \leq  n_{1,i},
$$
dual-charge-type
$$
\mathcal{R}= \left(r^{(1)}_{2},\ldots , r^{(s_{2})}_{2};r^{(1)}_{1},\ldots , r^{(s_{1})}_{1} \right),$$
where
$$
r^{(1)}_{i}\geq r^{(2)}_{i}\geq \ldots \geq  r^{(s_{i})}_{i}\geq 0 
$$
and color-type
$$ \left(r_{2},r_{1}\right),$$
where 
$$
r_i=\sum_{p=1}^{r_{i}^{(1)}}n_{p,i}=\sum^{s_{i}}_{t=1}r^{(t)}_{i} \ \ \text{and} \ \ s_{i}\in \mathbb{N},
$$
(cf. \cite{Bu1}--\cite{Bu2} and \cite{G}) if for every color $\mathcal{R}$ and $\mathcal{R}'$ are mutually conjugate partitions of $r_i$ (cf. \cite{A1}). We use the same terminology for the products of generating functions. 

We assume that all monomial factors are sorted so that energies of quasi-particles of the same color and the same charge form an increasing sequence of integers from right to left. We compare charge-type $\mathcal{R}'$ and $\overline{\mathcal{R}'}$, where  $\overline{\mathcal{R}'}=\left(\overline{n}_{\overline{r}_{2}^{(1)},2}, \ldots ,\overline{n}_{1,1}\right)$, so that we compare their charges from right to left, i.e. we write $\mathcal{R}'<\overline{\mathcal{R}'}$ if there is $u \in \mathbb{N}$, such that $n_{1,i}=\overline{n}_{1,i}, n_{2,i}=\overline{n}_{2,i},\ldots , n_{u-1,i}=\overline{n}_{u-1,i},$ and 
$u=\overline{r}_{i}^{(1)}+1$ or $n_{u,i}<\overline{n}_{u,i}$. 

We compare two monomials $b$ and $\overline{b}$ by comparing first their charge-types $\mathcal{R}'$ and $\overline{\mathcal{R}'}$ and then their sequences of energies $\left(m_{r_{2}^{(1)},2},\ldots , m_{1,1}\right)$ and $\left(\overline{m}_{\overline{r}_{2}^{(1)},2}, \ldots
,\overline{m}_{1,1}\right)$ (in a similar way as charge-types, again starting from color $i=1$): 
$$
b<\overline{b} \ \ \text{if} \ \ \left\{\begin{array}{ccc} \mathcal{R}'<\overline{\mathcal{R}'},&& \\\mathcal{R}'=\overline{\mathcal{R}'}& \text{and} &\left(m_{r_{2}^{(1)},2},\ldots , m_{1,1}\right)
<
\left(\overline{m}_{\overline{r}_{2}^{(1)},2}, \ldots
,\overline{m}_{1,1}\right).   \end{array}\right. 
$$

\subsection{Relations among quasi-particles}
\label{sec:4}
On a standard module $L(k\Lambda_0)$, we have vertex operator algebra relations
\begin{equation}\label{eq:6}
x_{(k+1)\alpha_1}(z)=0,
\end{equation}
\begin{equation}\label{eq:7}
x_{(3k+1)\alpha_2}(z)=0,
\end{equation}
\begin{equation}\label{eq:9}
x_{n\alpha_i}(z)v_{L(k\Lambda_0)} \in W_{L(k\Lambda_0)}\left[\left[z\right]\right],
\end{equation}
and 
\begin{equation}\label{eq:10} 
x_{n\alpha_i}(m)v_{L(k\Lambda_0)} = 0, \ \ \text{for} \ \ m > -n,
\end{equation}
when $n \leq k$ for $i=1$ and $n \leq 3k$ for $i=2$, 
(see \cite{LL}, \cite{MP}).

In reducing the set $Uv_{L(k\Lambda_0)}$ to the spanning set we use relations for a sequence of monochromatic monomial vectors (see Lemma 2.2.1 in \cite{Bu1}, or \cite{JP}, \cite{G}, \cite{F}) 
$$
x_{n\alpha_i}(m)x_{n'\alpha_i}(m')v_{L(k\Lambda_0)},x_{n\alpha_i}(m-1)x_{n'\alpha_i}(m'+1)v_{L(k\Lambda_0)}, \ldots $$
$$ \ldots  , x_{n\alpha_i}(m-2n+1)x_{n'\alpha_i}(m'+2n-1)v_{L(k\Lambda_0)},
$$
colored with color $i$ and with charge-type $(n,n')$, where $n< n'$, which we express as a (finite) linear combination of monomial vectors 
\begin{equation}\label{eq:11} 
x_{n\alpha_i}(j)x_{n'\alpha_i}(j')v_{L(k\Lambda_0)} \ \ \text{such that} \ \ j \leq m- 2n \ \ \text{and} \ \ j'\geq m'+2n
\end{equation}
and monomial vectors with a factor quasi-particle $x_{(n'+1)\alpha_i}(j_1)$, $j_1 \in \mathbb{Z}$.   

In the case when $n = n'$ monomials 
$$
x_{n\alpha_i}(m)x_{n\alpha_i}(m') \ \  \text{with} \ \ m'-2n< m \leq   m'
$$
can be expressed as a linear combination of monomials 
\begin{equation}\label{eq:12} 
x_{n\alpha_i}(j)x_{n\alpha_i}(j') \ \ \text{with} \ \ j\leq j'-2n\end{equation} 
and monomials with quasi-particle $x_{(n+1)\alpha_i}(j_1)$, $j_1 \in \mathbb{Z}$ (see Corollary 2.2.2 in \cite{Bu1}, or \cite{JP}, \cite{G}, \cite{F}).

Next, we consider products of quasi-particles colored with different colors. First, from commutation formulas (\ref{eq:2}) and (\ref{eq:3}) and induction on $n, n' \in \mathbb{N}$ follows 
\begin{lemma}\label{lem:1}  
Let $n\leq 3k, \ n'\leq k$ be fixed. We have:
\begin{itemize}
	\item [a)]  
$x_{\alpha_1}(0)x^{n}_{\alpha_2}(-1)v_{L(k\Lambda_0)}=-nx^{n-1}_{\alpha_2}(-1)x_{\alpha_1+\alpha_2}(-1)v_{L(k\Lambda_0)}+$ 
$$+\binom{n}{2} x^{n-2}_{\alpha_2}(-1)x_{\alpha_1+2\alpha_2}(-2)v_{L(k\Lambda_0)}-\binom{n}{3}x^{n-3}_{\alpha_2}(-1)x_{\alpha_1+3\alpha_2}(-3)v_{L(k\Lambda_0)};$$
	\item [b)] $x_{\alpha_1}(1)x^{n}_{\alpha_2}(-1)v_{L(k\Lambda_0)}=\binom{n}{2}x^{n-2}_{\alpha_2}(-1)x_{\alpha_1+2\alpha_2}(-1)v_{L(k\Lambda_0)}-$
$$-\binom{n}{3}x^{n-3}_{\alpha_2}(-1)x_{\alpha_1+3\alpha_2}(-2)v_{L(k\Lambda_0)};$$
	\item [c)] 
$x_{\alpha_1}(2)x^{n}_{\alpha_2}(-1)v_{L(k\Lambda_0)}=-\binom{n}{3}x^{n-3}_{\alpha_2}(-1)x_{\alpha_1+3\alpha_2}(-1)v_{L(k\Lambda_0)}$;
	\item [d)] 
$x_{\alpha_1}(j)x^{n}_{\alpha_2}(-1)v_{L(k\Lambda_0)}=0$, where $j \geq 3$;
	\item [e)] 
$x_{\alpha_2}(0)x^{n'}_{\alpha_1}(-1)v_{L(k\Lambda_0)}=n'x^{n'-1}_{\alpha_1}x_{\alpha_1+\alpha_2}(-1)v_{L(k\Lambda_0)}$;
	\item [f)] 
$x_{\alpha_2}(j)x^{n'}_{\alpha_1}(-1)v_{L(k\Lambda_0)}=0$, where $j \geq 1$.
	\end{itemize}
$$\ \ \ \ \ \ \ \ \ \ \ \ \ \ \ \ \ \ \ \ \ \ \ \ \ \ \ \ \ \ \ \ \ \ \ \ \ \ \ \ \ \ \ \ \ \ \ \ \ \ \ \ \ \ \ \ \ \ \ \ \ \ \ \ \ \ \ \ \ \ \ \ \ \ \ \ \ \ \ \ \ \ \ \ \ \ \ \ \ \ \ \ \ \ \ \ \ \ \ \ \ \ \ \ \ \ \ \ \square$$
\end{lemma}
Using the previous lemma follows relation among quasi-particles of different colors: 
\begin{lemma}
Let $n_{1}\leq k, \ n_2\leq 3k$. One has
\begin{equation}\label{eq:13} 
(z_{1}-z_{2})^{\text{min}\left\{3n_1,n_2\right\}}x_{n_1\alpha_{1}}(z_{1})x_{n_{2}\alpha_{2}}(z_{2})=(z_{1}-z_{2})^{\text{min}\left\{3n_1,n_2\right\}}x_{n_{2}\alpha_{2}}(z_{2})x_{n_1\alpha_{1}}(z_{1}).\end{equation}
\end{lemma}
\begin{proof}
Note, that from commutator formula for vertex operators (\ref{eq:4}), statements a), b), c) and d) of Lemma~\ref{lem:1} and properties of 
$\delta$-function we have 
\begin{equation}\label{eq:14} 
(z_{1}-z_{2})^{3}x_{\alpha_{1}}(z_{1})x_{n_2\alpha_{2}}(z_{2})=(z_{1}-z_{2})^{3}x_{n_{2}\alpha_{2}}(z_{2})x_{\alpha_{1}}(z_{1}).\end{equation}
In a similar way, using e) and f) parts of Lemma~\ref{lem:1} we have
\begin{equation}\label{eq:15}
(z_{1}-z_{2})x_{n_1\alpha_{1}}(z_{1})x_{\alpha_{2}}(z_{2})=(z_{1}-z_{2})x_{\alpha_{2}}(z_{2})x_{n_1\alpha_{1}}(z_{1}).\end{equation}
Now, from (\ref{eq:14}) and (\ref{eq:15}) follows the lemma.
\end{proof}

By using derived relations we can define the set of quasi-particle monomials which generate our bases (acting on the highest weight vectors)   
$$
B_{W_{L(k\Lambda_{0})}}= \bigcup_{\substack{n_{r_{1}^{(1)},1}\leq \ldots \leq n_{1,1}\leq 
k\\\substack{n_{r_{2}^{(1)},2}\leq \ldots \leq n_{1,2}\leq 3k}}}\left(\text{or, equivalently,} \ \ \ 
\bigcup_{\substack{r_{1}^{(1)}\geq \cdots\geq r_{1}^{(k)}\geq 0\\\substack{ r_{2}^{(1)}\geq \cdots\geq r_{2}^{(3k)}\geq 
0}}}\right)
$$
$$
\left\{b\right.= b(\alpha_{2}) b(\alpha_{1})
=x_{n_{r_{2}^{(1)},2}\alpha_{2}}(m_{r_{2}^{(1)},2})\cdots x_{n_{1,2}\alpha_{2}}(m_{1,2}) x_{n_{r_{1}^{(1)},1}\alpha_{1}}(m_{r_{1}^{(1)},1})\cdots  x_{n_{1,1}\alpha_{1}}(m_{1,1}):
$$
$$
\left|
\begin{array}{l}
m_{p,1}\leq  -n_{p,1}- \sum_{p>p'>0} 2 \ \text{min}\{n_{p,1}, n_{p',1}\},  \ 1\leq  p\leq r_{1}^{(1)};\\
m_{p+1,1} \leq   m_{p,1}-2n_{p,1} \  \text{if} \ n_{p+1,1}=n_{p,1}, \ 1\leq  p\leq r_{1}^{(1)}-1;\\
m_{p,2}\leq  -n_{p,2} + \sum_{q=1}^{r_{1}^{(1)}}\text{min}\left\{ 3n_{q,1},n_{p,2 }\right\} - \sum_{p>p'>0} 2 \ \text{min}\{n_{p,2}, n_{p',2}\}, \  1\leq  p\leq r_{2}^{(1)};\\
m_{p+1,2}\leq  m_{p,2}-2n_{p,2} \  \text{if} \ n_{p,2}=n_{p+1,2}, \  1\leq  p\leq r_{2}^{(1)}-1 %\\
\end{array}\right\}.
$$
The condition on energies of quasi-particles colored with color $i=2$ contains a part which follows from relation \ref{eq:13}. 
The other conditions on energies which follow from relations (\ref{eq:6} -- \ref{eq:12}) are similar to difference conditions as in the case of $B_2^{(1)}$.  

Now, we can state the Proposition~\ref{pro:1}, whose proof follows closely \cite{G}:
\begin{proposition}\label{pro:1} 
The set $\mathfrak{B}_{W_{L(k\Lambda_{0})}}=\left\{bv_{L(k\Lambda_{0})}:b \in B_{W_{L (k\Lambda_{0})}}\right\}$ spans the principal subspace $W_{L(k\Lambda_{0})}$.
$$\ \ \ \ \ \ \ \ \ \ \ \ \ \ \ \ \ \ \ \ \ \ \ \ \ \ \ \ \ \ \ \ \ \ \ \ \ \ \ \ \ \ \ \ \ \ \ \ \ \ \ \ \ \ \ \ \ \ \ \ \ \ \ \ \ \ \ \ \ \ \ \ \ \ \ \ \ \ \ \ \ \ \ \ \ \ \ \ \ \ \ \ \ \ \ \ \ \ \ \ \ \ \ \ \ \ \ \ \square$$
\end{proposition}

In the rest of this section we consider the proof of linear independence of the set $\mathfrak{B}_{W_{L(k\Lambda_{0})}}$. First, we introduce the properties of operators on a standard module level 1, which we will use in our proof.

\subsection{Projection \texorpdfstring{$\pi_{\mathcal{R}}$}{pi{R}}}
\label{sec:5}
Let $k>1$. We realize the principal subspace $W_{L(k\Lambda_{0})}$ as a subspace of the tensor product $W_{L(\Lambda_{0})}^{\otimes k}\subset  L(\Lambda_{0})^{\otimes k}$, where
$$
v_{L(k\Lambda_0)}=\underbrace{v_{L(\Lambda_{0})} \otimes \cdots \otimes v_{L(\Lambda_{0})}}_{k \ \text{factors}}$$
is the highest weight vector.

For a chosen dual-charge-type 
$$
\mathfrak{R}=\left( r_{2}^{(1)}, \ldots ,r_{2}^{(3k)}; r_{1}^{(1)}, \ldots , r_{1}^{(k)}\right),
$$
denote by $\pi_{\mathfrak{R}}$ the projection of principal subspace $W_{L(k\Lambda_{0})}$ 
to the subspace
$$ {W_{L(\Lambda_0)}}_{(\mu^{(k)}_{2};r_{1}^{(k)})}\otimes \cdots \otimes  {W_{L(\Lambda_0)}}_{(\mu^{(1)}_{2};r_{1}^{(1)})},
$$
where  
${W_{L(\Lambda_0)}}_{(\mu^{(t)}_{2};r_{1}^{(t)})}$
 is a $\mathfrak{h}$-weight subspace of weight $\mu^{(t)}_{2}\alpha_2+ r_{1}^{(t)}\alpha_1 \in Q$ with
$$
\mu^{(t)}_{2}=r^{(3t)}_{2}+ r^{(3t-1)}_{2}+ r^{(3t-2)}_{2},
$$
for every $1 \leq  t \leq k$. 

We shall denote by the same symbol $\pi_{\mathfrak{R}}$ the generalization of this projection to the space of formal series with coefficients in $W_{L(\Lambda_{0})}^{\otimes k}$.  
Let
\begin{equation}\label{eq:16} 
x_{n_{r_{2}^{(1)},2}\alpha_{2}}(z_{r_{2}^{(1)},2}) \cdots     x_{n_{1,2}\alpha_{2}}(z_{1,2})x_{n_{r_{1}^{(1)},1}\alpha_{1}}(z_{r_{1}^{(1)},1})\cdots  x_{n_{1,1}\alpha_{1}}(z_{1,1})
\end{equation}
be a generating function of the chosen dual-charge-type $\mathfrak{R}$ and the corresponding charge-type $\mathfrak{R}'$. Then, from relations (\ref{eq:6}) and (\ref{eq:7}) and definition of the action of Lie algebra on the modules, follows that the projection of the generating function (\ref{eq:16}) is
\begin{equation}\label{eq:39}
\pi_{\mathfrak{R}} x_{n_{r_{2}^{(1)},2}\alpha_{2}}(z_{r_{2}^{(1)},2})\cdots  x_{n_{1,1}\alpha_{1}}(z_{1,1}) \ v_{L(k\Lambda_{0})}
\end{equation}
$$=\text{C} \ x_{n^{(k)}_{r^{(3k-2)}_{2},2}\alpha_{2}}(z_{r_{2}^{(3k-2)},2})\cdots                x_{n^{(k)}_{r^{(3k-1)}_{2},2}\alpha_{2}}(z_{r^{(3k-1)}_{2},2})\cdots x_{n^{(k)}_{r^{(3k)}_{2},2}\alpha_{2}}(z_{r^{(3k)}_{2},2})\cdots$$
$$ \ \ \ \ \ \ \ \ \ \ \ \ \ \ \ \ \ \ \ \ \ \ \ \ \ \ \ \ \ \ \ \ \ \ \ \  \cdots x_{n^{(k)}_{1,2}\alpha_{2}}(z_{1,2}) x_{n^{(k)}_{r^{(k)}_{1},1}\alpha_{1}}(z_{r_{1}^{(k)},1})\cdots   x_{n^{(k)}_{1,1}\alpha_{1}}(z_{1,1}) \ v_{L(\Lambda_{0})}$$
$$ \ \ \ \ \ \ \ \ \ \ \ \ \ \ \ \ \otimes \cdots \otimes$$
$$
\otimes x_{n_{r^{(1)}_{2},2}^{(1)}\alpha_{2}}(z_{r_{2}^{(1)},2})\cdots  x_{n_{r^{(2)}_{2},2}^{(1)}\alpha_{2}}(z_{r_{2}^{(2)},2})\cdots  x_{n_{r^{(3)}_{2},2}^{(1)}\alpha_{2}}(z_{r_{2}^{(3)},2})\cdots 
  x_{n_{1,2}^{(1)}\alpha_{2}}(z_{1,2})$$
$$ \ \ \ \ \ \ \ \ \ \ \ \ \ \ \ \ \ \ \ \ \ \ \ \ \ \ \ \ \ \ \ \ \ \ \ \ x_{n_{r^{(1)}_{1},1}^{(1)}\alpha_{1}}(z_{r_{1}^{(1)},1})\cdots     x_{n{_{1,1}^{(1)}\alpha_{1}}}(z_{1,1}) \ v_{L(\Lambda_{0})},$$
where $\text{C} \in \mathbb{C}^{*}$, 
$$
0 \leq  n^{(t)}_{p,2}\leq  3, \ \ n^{(1)}_{p,2}\geq  n^{(2)}_{p,2} \geq  \ldots  \geq  n^{(k-1)}_{p,2} \geq  n^{(k)}_{p,2}, \ n_{p,2}=\sum_{t=1}^k n^{(t)}_{p,2},
$$
for every every $p$, $1 \leq p \leq r_{2}^{(1)}$, so that at most one $n^{(t)}_{p,2}$ ($1 \leq t \leq k$) can be $1$ or 2 
and 
$$
0 \leq  n^{(t)}_{p,1} \leq 1, \ \ 1 \leq t \leq k, \ n^{(1)}_{p,1}\geq n^{(2)}_{p,1}\geq \ldots \geq  n^{(k-1)}_{p,1}\geq  n^{(k)}_{p,1}, \ 
n_{p,1}=\sum_{t=1}^k n^{(t)}_{p,1},
$$
for every every $p$, $1 \leq p \leq r_{1}^{(1)}$.

\begin{example}\label{exa:1} 
In the case when $k=2$ the projection $\pi_{\mathfrak{R}}$, where $\mathfrak{R}=\left(6,5,4,3,2,1;3,2\right)$, of generating function 
$$
x_{\alpha_{2}}(z_{6,2})x_{2\alpha_{2}}(z_{5,2})x_{3\alpha_{2}}(z_{4,2})x_{4\alpha_{2}}(z_{3,2})x_{5\alpha_{2}}(z_{2,2})x_{6\alpha_{2}}(z_{1,2})x_{\alpha_{1}}(z_{3,1})x_{2\alpha_{1}}(z_{2,1})x_{2\alpha_{1}}(z_{1,1}) 
$$
on 
$${W_{L(\Lambda_0)}}_{(6;2)}\otimes  {W_{L(\Lambda_0)}}_{(15;3)}$$
can be represented graphically as in the Figure~\ref{fig:1}, where at most one generating function of color $i=1$ is placed on every tensor factor and at most three generating functions of color $i=2$ are placed on every tensor factor.
\end{example}

\bigskip
\begin{figure}[ht]
\begin{center}
\setlength{\unitlength}{8mm}
\begin{picture}(10,5)
\linethickness{0.1mm}
\multiput(0,0)(1,0){1}%
{\line(0,1){1}}
\multiput(1,0)(1,0){1}%
{\line(0,1){2}}
\multiput(2,0)(1,0){1}%
{\line(0,1){3}}
\multiput(3,0)(1,0){1}%
{\line(0,1){4}}
\multiput(4,0)(1,0){1}%
{\line(0,1){5}}
\multiput(5,0)(1,0){1}%
{\line(0,1){6}}
\linethickness{0.75mm}
\multiput(6,0)(1,0){1}%
{\line(0,1){6}}
\linethickness{0.1mm}
\multiput(7,0)(1,0){1}%
{\line(0,1){1}}
\multiput(8,0)(1,0){1}%
{\line(0,1){1}}
\multiput(9,0)(1,0){1}%
{\line(0,1){1}}
\multiput(0,0)(0,1){2}%
{\line(1,0){9}}
\multiput(1,2)(0,1){1}%
{\line(1,0){5}}
\linethickness{0.75mm}
\multiput(2,3)(0,1){1}%
{\line(1,0){7}}
\linethickness{0.1mm}
\multiput(3,4)(0,1){1}%
{\line(1,0){3}}
\multiput(4,5)(0,1){1}%
{\line(1,0){2}}
\multiput(5,6)(0,1){1}%
{\line(1,0){1}}
\multiput(7,3)(0,1){1}%
{\line(0,1){1}}
\multiput(8,3)(0,1){1}%
{\line(0,1){1}}
\multiput(9,3)(0,1){1}%
{\line(0,1){1}}
\multiput(7,4)(0,1){1}%
{\line(1,0){2}}
\put(9.99,0.5){\scriptsize{\footnotesize{$v_{L(\Lambda_0)}$}}}
\put(9.99,3.5){\scriptsize{\footnotesize{$v_{L(\Lambda_0)}$}}}
\end{picture}
\bigskip
\caption{Example~\ref{exa:1}}
\label{fig:1}%\label{slika}
\end{center}
\end{figure}

We define the projection of monomial vector $bv_{L(k\Lambda_{0})}$, with $b\in B_{W_{L(k\Lambda_0)}}$ 
colored with color-type $(r_{2},r_{1}),$ charge-type $\mathfrak{R}'$ and dual-charge-type $\mathfrak{R}$
\begin{equation}\label{eq:17} 
b=x_{n_{r_{2}^{(1)},2}\alpha_{2}}(m_{r_{2}^{(1)},2})\cdots      x_{n_{1,2}\alpha_{2}}(m_{1,2})x_{n_{r_{1}^{(1)},1}\alpha_{1}}(m_{r_{1}^{(1)},1})\cdots  x_{n_{1,1}\alpha_{1}}(m_{1,1})
\end{equation}
as a coefficient of the projection of the generating function  (\ref{eq:39}) which we denote as $$\pi_{\mathfrak{R}}bv_{L(k\Lambda_{0})}.$$
If $\bar{b}\in B_{W_{L(k\Lambda_0)}}$ is a monomial of charge-type $(\bar{n}_{{\bar{r}}_{2}^{(1)},2}, \ldots ,  \bar{n}_{1,2}; \bar{n}_{{\bar{r}}_{1}^{(1)},1}, \ldots , \bar{n}_{1,1}),$ dual-charge-type 
$\bar{\mathfrak{R}}=\left( {\bar{r}}_{2}^{(1)}, \ldots ,{\bar{r}}_{2}^{(3k)};{\bar{r}}_{1}^{(1)}, \ldots , {\bar{r}}_{1}^{(k)}\right)$ and such that 
$$
b<\bar{b},
$$
then, from the definition of projection, follows that 
$$
\pi_{\mathfrak{R}}\bar{b}v_{L(k\Lambda_0)}=0.
$$
We will use this property of projection $\pi_{\mathfrak{R}}$ in the proof of linear independence.
 
\subsection{\textbf{Operator \texorpdfstring{$A_{\theta}$}{Atheta}}}
\label{sec:6}
Denote by $A_{\theta}$ the coefficient of an intertwining operator 
$x_{\theta}(z)$ 
$$
A_{\theta}=\text{Res}_z z^{-1} \ x_{\theta}(z) = x_{\theta} (-1)
$$
which commutes with the action of $\mathcal{L}(\mathfrak{n}_+)$ and such that
\begin{equation}\label{eq:18} 
A_{\theta}v_{L(\Lambda_0)}=x_{\theta}(-1)v_{L(\Lambda_0)}.
\end{equation}  
We act with operator 
$$
1\otimes\cdots \otimes  A_{\theta} \otimes \underbrace{1 \otimes \cdots \otimes 1}_{s-1 \ \text{factors}}, \ \ s\leq k
$$
on the vector $bv_{L(k\Lambda_0)}\in \mathcal{B}_{W_{L(k\Lambda_0)}}$, where quasi-particle monomial $b$ is as in (\ref{eq:17}).
From the definition of projection, it follows that vector 
$$
\left(1\otimes\cdots \otimes 1\otimes A_{\theta} \otimes 1 \otimes \cdots \otimes 1\right)(\pi_{\mathfrak{R}}bv_{L(k\Lambda_0)})
$$
is the coefficient of  
\begin{equation}\label{eq:19} 
(1\otimes\cdots \otimes A_{\theta} \otimes 1 \otimes \cdots \otimes  1)\pi_{\mathfrak{R}}x_{n_{r_{2}^{(1)},2}\alpha_{2}}(z_{r_{2}^{(1)},2})\cdots x_{s\alpha_{1}}(z_{1,1})v_{L(k\Lambda_0)}.
\end{equation}
From (\ref{eq:18}) it follows that in the $s$-th tensor row of (\ref{eq:19}) we have  
$$\otimes  x_{n^{(s)}_{r^{(3s-2)}_{2},2}\alpha_{2}}(z_{r_{2}^{(3s-2)},2})\cdots   x_{n^{(s)}_{r^{(3s-1)}_{2},2}\alpha_{2}}(z_{r^{(3s-1)}_{2},2})\cdots   x_{n^{(s)}_{r^{(3s)}_{2},2}\alpha_{2}}(z_{r^{(3s)}_{2},2})\cdots  x_{n^{(s)}_{1,2}\alpha_{l}}(z_{1,2})$$
\begin{equation}\label{eq:20} 
x_{n^{(s)}_{r^{(s)}_{1},1}\alpha_{1}}(z_{r_{1}^{(s)},1})\cdots  x_{\alpha_{1}}(z_{1,1})x_{\theta}(-1)v_{L(\Lambda_0)}\otimes \cdots ,\end{equation}
where $ 0 \leq  n^{(s)}_{p,1} \leq 1$, for $1\leq  p \leq  r^{(s)}_{1}$ and $ 0 \leq  n^{(s)}_{p,2} \leq 3$, for $1\leq p \leq    r^{(3s-2)}_{2}$.  

\subsubsection{\textbf{Operators \texorpdfstring{$e_{\alpha }$}{ealpha}}}\label{S66B}
\label{sec:7} 
For every root $\alpha  \in R$, we define on the level $1$ standard module $L(\Lambda_0)$, the ``Weyl group translation'' operator  $e_{\alpha}$ by
$$
 e_{\alpha}=\exp\  x_{-\alpha}(1)\exp\  (- x_{\alpha}(-1))\exp\  x_{-\alpha}(1) \exp\ x_{\alpha}(0)\exp\   (-x_{-\alpha}(0))\exp\ x_{\alpha}(0),$$
(for properly normalized root vectors, cf. \cite{K}). Then on $L(\Lambda_0)$ we have
\begin{equation}\label{eq:21} 
 e_{\alpha}v_{L(\Lambda_0)}=-x_{\alpha}(-1)v_{L(\Lambda_0)}
\end{equation}
\begin{equation}\label{eq:22} 
x_{\beta}(j)e_{\alpha}=e_{\alpha}x_{\beta}(j-\beta(\alpha\sp\vee)), \ \ \beta \in R, \ \ j \in \mathbb{Z}.
\end{equation}
For $\alpha=\theta$, from (\ref{eq:21}) and (\ref{eq:22}), it follows that we can (\ref{eq:20}) write as  
$$
\cdots \otimes x_{n^{(s)}_{r^{(3s-2)}_{l},2}\alpha_{2}}(z_{r_{2}^{(3s-2)},2})\cdots   x_{n^{(s)}_{r^{(3s-1)}_{2},2}\alpha_{2}}(z_{r^{(3s-1)}_{2},2})\cdots           x_{n^{(s)}_{r^{(3s)}_{2},2}\alpha_{2}}(z_{r^{(3s)}_{2},2})\cdots  x_{n^{(s)}_{1,2}\alpha_{l}}(z_{1,2})$$
$$
x_{n^{(s)}_{r_{1}^{(k)},1}\alpha_{1}}(z_{r_{1}^{(k)},1})z_{r_{1}^{(k)},1}\cdots x_{\alpha_{1}}(z_{1,1})z_{1,1}v_{L(\Lambda_{0})}\otimes\cdots.$$
By taking the corresponding coefficients, we have
$$
(1\otimes\cdots \otimes 1\otimes A_{\theta} \otimes 1 \otimes \cdots \otimes 1)\pi_{\mathfrak{R}}bv_{L(k\Lambda_0)}=(1\otimes\cdots \otimes 1\otimes      e_{\theta} \otimes 1 \otimes \cdots \otimes 1)\pi_{\mathfrak{R}}b^{+}v_{L(k\Lambda_0)}
$$
where 
\begin{align}\nonumber
b^{+}=b^+(\alpha_{2}) b^{+}(\alpha_{1})&=b(\alpha_{2})x_{n_{r_{1}^{(1)},1}\alpha_{1}}(m_{r_{1}^{(1)},1}+1)\cdots x_{s\alpha_{1}}(m_{1,1}+1).&
\end{align}

Now, let $\alpha=\alpha_1$. We consider the projection $\pi_{\mathfrak{R}}bv_{L(k\Lambda_{0})}$ of the monomial vector $bv_{L(k\Lambda_{0})}$ where $b \in B_{W_{L(k\Lambda_{0})}}$ is a monomial   
\begin{equation}\label{eq:23}
b=b(\alpha_{2})b(\alpha_{1})x_{s\alpha_{1}}(-s)\end{equation}
$$=x_{n_{r^{(1)}_{2},2}\alpha_{2}}(m_{r^{(1)}_{2},2})\cdots     x_{n_{1,2}\alpha_{2}}(m_{1,2})x_{n_{r^{(1)}_{1},1}\alpha_{1}}(m_{r^{(1)}_{1},1})\cdots  x_{n_{2,1}\alpha_{1}}(m_{2,1})x_{s\alpha_{1}}(-s),$$
of dual-charge-type 
$$
\mathfrak{R}=\left(r^{(1)}_{2},\ldots, r^{(3k)}_{2};r^{(1)}_{1},\ldots, r_1^{(s)},0 \ldots, 0\right).
$$
The projection is a coefficient of the generating function
$$\pi_{\mathfrak{R}}x_{n_{r_{2}^{(1)},2}\alpha_{2}}(z_{r_{2}^{(1)},2})\cdots
 x_{n_{1,2}\alpha_{2}}(z_{1,2})x_{n_{r_{1}^{(1)},1}\alpha_{1}}(z_{r_{1}^{(1)},1}) 
\cdots  x_{n_{2,1}\alpha_{1}}(z_{2,1})$$
$$\left(v_{L(\Lambda_{0})}\otimes
 \cdots \otimes  v_{L(\Lambda_{0})}\otimes x_{\alpha_{1}}(-1)v_{L(\Lambda_{0})}\otimes \cdots \otimes  x_{\alpha_{1}}(-1)v_{L(\Lambda_{0})}\right)$$
$$=Cx_{n^{(k)}_{r^{(3k-2)}_{2},2}\alpha_{2}}(z_{r_{2}^{(3k-2)},2}) \cdots 
 x_{n^{(k)}_{r^{(3k-1)}_{2},2}\alpha_{2}}(z_{r^{(3k-1)}_{2},2})\cdots x_{n^{(k)}_{r^{(3k)}_{2},2}\alpha_{2}}(z_{r^{(3k)}_{2},2})\cdots
 x_{n^{(k)}_{1,2}\alpha_{2}}(z_{1,2})v_{L(\Lambda_{0})}$$
$$\otimes \cdots \otimes $$
$$\otimes x_{n^{(s)}_{r^{(3s-2)}_{2},2}\alpha_{2}}(z_{r_{2}^{(3s-2)},2}) \cdots 
 x_{n^{(s)}_{r^{(3s-1)}_{2},2}\alpha_{2}}(z_{r^{(3s-1)}_{2},2})\cdots x_{n^{(s)}_{r^{(3s)}_{2},2}\alpha_{2}}(z_{r^{(3s)}_{2},2})\cdots  x_{n^{(s)}_{1,2}\alpha_{2}}(z_{1,2})$$
$$ x_{n^{(s)}_{r^{(s)}_{1},1}\alpha_{1}}(z_{r_{1}^{(s)},1})\cdots    x_{n^{(s)}_{2,1}\alpha_{1}}(z_{2,1})e_{\alpha_{1}}v_{L(\Lambda_{0})}$$
$$\otimes \cdots \otimes $$
$$\otimes x_{n_{r^{(1)}_{2},2}^{(1)}\alpha_{2}}(z_{r_{2}^{(1)},2})\cdots 
 x_{n_{r^{(2)}_{2},2}^{(1)}\alpha_{2}}(z_{r_{2}^{(2)},2})\cdots x_{n_{r^{(3)}_{2},2}^{(1)}\alpha_{2}}(z_{r_{2}^{(3)},2})\cdots x_{n_{2,2}^{(1)}\alpha_{2}}(z_{2,2})
x_{n_{1,2}^{(1)}\alpha_{2}}(z_{1,2})$$
$$x_{n_{r^{(1)}_{1},1}^{(1)}\alpha_{1}}(z_{r_{1}^{(1)},1})\cdots   x_{n{_{2,1}^{(1)}\alpha_{1}}}(z_{2,1})e_{\alpha_{1}}v_{L(\Lambda_{0})}.$$
Now, if we shift $1 \otimes \cdots \otimes  e_{\alpha_{1}} \otimes e_{\alpha_{1}} \otimes \cdots \otimes  e_{\alpha_{1}}$ all the way to left using commutation relations (\ref{eq:22}) 
we get
$$
\pi_{\mathfrak{R}}bv_{L(k\Lambda_{0})}=(1 \otimes \cdots \otimes  e_{\alpha_{1}} \otimes e_{\alpha_{1}} \otimes \cdots \otimes  e_{\alpha_{1}})\pi_{\mathfrak{R}^{-}} b'v_{k\Lambda_{0}}, 
$$
where $b'$ is quasi-particle monomial
\begin{equation}\label{eq:24}
b'=b'(\alpha_{2})b'(\alpha_{1})\end{equation}
$$=x_{n_{r^{(1)}_{2},2}\alpha_{2}}(m_{r^{(1)}_{2},2}-n^{(1)}_{r_{2}^{(1)},2}-\cdots-n^{(s)}_{r_{2}^{(1)},2})\cdots x_{n_{1,2}\alpha_{2}}(m_{1,2}-n^{(1)}_{1,2}-\cdots-n^{(s)}_{1,2})$$
$$ \ \ \ \ \ \ \ \ \ \ \ \ \ \ \ \ \ \ \ \ \ \ \ \ \ \ \  x_{n_{r^{(1)}_{1},1}\alpha_{1}}(m_{r^{(1)}_{1},1}+2n_{r^{(1)}_{1}})\cdots   x_{n_{2,1}\alpha_{1}}(m_{2,1}+2n_{1,2})$$
$$=x_{n_{r^{(1)}_{2},2}\alpha_{2}}(m'_{r^{(1)}_{2},2})\cdots 
 x_{n_{1,2}\alpha_{2}}(m'_{1,2})x_{n_{r^{(1)}_{1},1}\alpha_{1}}(m'_{r^{(1)}_{1},1})\cdots x_{n_{2,1}\alpha_{1}}(m'_{2,1}),$$
$0 \leq  n^{(t)}_{p,2}\leq  3, \ \ 0 \leq  n^{(t)}_{p,1}\leq  1, \ \ 1 \leq p \leq r_i^{(s)}, 1\leq t \leq s$ of dual-charge-type
$$\mathfrak{R}^{-}=\left(r^{(1)}_{2},\ldots, r^{(3s)}_{2};r^{(1)}_{1}-1,\ldots, r_1^{(s)}-1\right).
$$

\begin{proposition}\label{pro:2}
Monomial $b'$ (\ref{eq:24}) is an element of the set $B_{W_{L(k\Lambda_{0})}}$. 
\end{proposition}
\begin{proof}
We will prove that $m'_{p,2}$, $2  \leq p  \leq r_{1}^{(1)}$ and $1  \leq p  \leq r_{2}^{(1)}$ 
satisfy the same difference conditions as energies of quasi-particle monomials from the set $B_{W_{L(k\Lambda_{0})}}$. We will consider only energies $m'_{p,2}$, since for the color $i=1$ the proof is similar as in the case of $B_2^{(1)}$. We have two cases:
\begin{enumerate}
\item [1)] if $n_{p,2} \geq 3s$, then we have:
\begin{align}\nonumber
m'_{p,2}&= m_{p,2}-3s&\\
\nonumber
& \leq -n_{p,2} + \sum_{q=1}^{r_{1}^{(1)}}\text{min}\left\{ 3n_{q,1},n_{p,2 }\right\} - \sum_{p>p'>0} 2 \ \text{min}\{n_{p,2}, n_{p',2}\}-3s;&
\end{align}
\begin{align}\nonumber
m'_{p+1,2}&= m_{p+1,2}-3s&\\
\nonumber
& \leq m_{p,2} -2n_{p,2 }-3s&\\
\nonumber
& = m'_{p,2}-2n_{p,2 } \ \ \text{when} \ \  n_{p,2 } =n_{p+1,2 } ;
\end{align}
\item[2)] if $n_{p,2} < 3s$, then we have:
\begin{align}\nonumber
m'_{p,2}&= m_{p,2}-n_{p,2}&\\
\nonumber
& \leq -n_{p,2} + \sum_{q=1}^{r_{1}^{(1)}}\text{min}\left\{ 3n_{q,1},n_{p,2 }\right\} - \sum_{p>p'>0} 2 \ \text{min}\{n_{p,2}, n_{p',2}\}-n_{p,2}&\\
\nonumber
& = -n_{p,2} + \sum_{q=2}^{r_{1}^{(1)}}\text{min}\left\{ 3n_{q,1},n_{p,2 }\right\} - \sum_{p>p'>0} 2 \ \text{min}\{n_{p,2}, n_{p',2}\};&
\end{align}
\begin{align}\nonumber
m'_{p+1,2}&= m_{p+1,2}-n_{p,2}&\\
\nonumber
& \leq m_{p,2} -2n_{p,2 }-n_{p,2}&\\
\nonumber
& = m'_{p,2}-2n_{p,2 } \ \ \text{when} \ \  n_{p,2 } =n_{p+1,2 } .&
\end{align}
\end{enumerate}
\end{proof}

\subsection{Proof of linear independence}
\label{sec:8}
We prove linear independence of the set $\mathfrak{B}_{W_{L(k\Lambda_0)}}$  by induction on charge-type of monomials from the set $B_{W_{L(k\Lambda_0)}}$. Then from the Proposition~\ref{pro:1} will follow 
  
\begin{theorem}\label{thm:1}
The set $\mathfrak{B}_{W_{L(k\Lambda_0)}}$ is a basis of the principal subspace $W_{L(k\Lambda_0)}$.
\end{theorem}

\begin{proof}
First consider a finite linear combination  
\begin{equation}\label{eq:25}
\sum_{a }
c_{a}b_av_{L(k\Lambda_{0})}=0
\end{equation}
of monomial vectors $b_av_{L(k\Lambda_{0})} \in  \mathfrak{B}_{W_{L(k\Lambda_0)}}$ of the same color-type $(r_{2}, r_{1})$. Denote by $b=b(\alpha_2)b(\alpha_1)x_{n_{1,1}\alpha_1}(j)$ the smallest monomial in (\ref{eq:25}) such that $c_{a}\neq 0$. Assume that $b$ is 
of charge-type  
\begin{equation}\label{eq:26} 
\mathfrak{R}'=\left(n_{r_{2}^{(1)},2}, \ldots , 
 n_{1,2};  n_{r_{2}^{(1)},1}, \ldots ,n_{1,1}\right)
\end{equation}
and a dual-charge-type 
$$\mathfrak{R}=\left( r_{2}^{(1)},\ldots , r_{2}^{(3k)};r_{1}^{(1)},\ldots , r_{1}^{(n_{1,1})}\right),
$$
which determines the projection $\pi_{\mathfrak{R}}$ on the vector space 
$${W_{L(\Lambda_{0})}}_{(\mu^{(k)}_{2}; 0)}\otimes \cdots  \otimes {W_{L(\Lambda_{0})}}_{(\mu^{(n_{1,1}+1)}_{2};0)}\otimes   
 {W_{L(\Lambda_{0})}}_{(\mu^{(n_{1,1})}_{2}; r_{1}^{(n_{1,1})})}\otimes  \cdots \otimes {W_{L(\Lambda_{0})}}_{(\mu^{(1)}_{2};r_{1}^{(1)})},$$
where 
$$\mu^{(t)}_{l}=r^{(3t-1)}_{2}+ r^{(3t)}_{2}+ r^{(3t-2)}_{2}, \ \ 1 \leq t \leq k.
$$
$\pi_{\mathfrak{R}}$ maps to zero all vectors $b_av_{L(k\Lambda_{0})}$ in (\ref{eq:25}) with monomials $b_a$ of larger charge-type's than $\mathfrak{R}'$. 
Now, in 
\begin{equation}\label{eq:27}
\sum_{a} c_a\pi_{\mathfrak{R}}b_{a}v_{L(k\Lambda_{0})}=0 
\end{equation}
we have a projection of $b_{a}v_{L(k\Lambda_{0})}$, where $b_{a}$ are of charge-type $\mathfrak{R}'$. 

On (\ref{eq:27}), we act with operators 
$1\otimes\cdots \otimes A_{\theta} \otimes \underbrace{1 \otimes \cdots \otimes 1}_{n_{1,1}-1 \ \text{factors}}$ and commute to the left with operators $1\otimes\cdots \otimes
e_{\theta} \otimes \underbrace{1 \otimes \cdots \otimes 1}_{n_{1,1}-1 \ \text{factors}}$  
until we get 
\begin{equation}\label{eq:28}%\label{LN3}
\sum_{a} c_a\pi_{\mathfrak{R}}b_a(\alpha_2)b_a(\alpha_1)x_{n_{1,1}\alpha_1}(-n_{1,1})v_{L(k\Lambda_{0})}=0. 
\end{equation}
Note, that in (\ref{eq:28}) we only have monomial vectors of charge-type $\mathfrak{R}'$ with quasi-particle $x_{n_{1,1}\alpha_1}(-n_{1,1})$, since operators used above at some point annihilate all other monomial vectors with $x_{n_{1,1}\alpha_1}(m_{1,1}), \ m_{1,1}> -j$.

From the consideration in previous subsection it follows that (\ref{eq:28}) can be written as
\begin{equation}\label{eq:29}
1\otimes\cdots \otimes e_{\alpha_{1}} \otimes \cdots \otimes e_{\alpha_{1}}\left(\sum_{a} c_a\pi_{\mathfrak{R^{-}}}b'_a(\alpha_2)b'_a(\alpha_1)v_{L(k\Lambda_{0})}\right)=0, 
\end{equation}
and after dropping out the invertible operator $1\otimes\cdots \otimes e_{\alpha_{1}} \otimes \cdots \otimes e_{\alpha_{1}}$ as
$$\sum_{a} c_a\pi_{\mathfrak{R^{-}}}b'_a(\alpha_2)b'_a(\alpha_1)v_{L(k\Lambda_{0})}=0
$$
where $b'_a(\alpha_2)b'_a(\alpha_1) \in B_{W_{L(k\Lambda_0)}}$ are quasi-particle monomials od dual-charge type
$$
\mathfrak{R}^{-}=\left( r_{2}^{(1)}, \ldots ,r_{2}^{(k)}; r_{1}^{(1)}-1, \ldots , r_{1}^{(n_{1,1})}-1\right),
$$
 with smaller charge-type from $\mathfrak{R}'$. 

We repeat the described processes, until we get 
\begin{equation}\label{eq:30}%\label{LN5}
\sum_{a} c_a\pi_{\mathfrak{R^{-}}}b_a(\alpha_2)v_{L(k\Lambda_{0})}=0,
\end{equation}
where monomial vectors $b_a(\alpha_2)v_{L(k\Lambda_{0})}$ are colored only with color $i=2$.

Similar as in the case of $B_2^{(1)}$ we will see vectors $b_a(\alpha_2)v_{L(k\Lambda_{0})}$ in (\ref{eq:30}) as elements of
\begin{equation}\label{eq:31} 
\underbrace{W_{L^A(3\Lambda_{0})}\otimes  \cdots  \otimes  W_{L^A(3\Lambda_{0})}}_{k \ \text{factors}},
\end{equation}
where $W_{L^A(3\Lambda_{0})}$ is the principal subspace of level 3 standard $\widetilde{sl}_{2}(\alpha_{2})$-module $L^A(3\Lambda_{0})$ with the highest weight vector $v_{L(\Lambda_0)}$. Denote by $\pi'_{\mathfrak{R^{-}}}$ the projection of (\ref{eq:31}) on 
\begin{equation}\label{eq:32} 
{W_{L^A(\Lambda_{0})}}_{(r^{(3k)}_{2})}\otimes {W_{L(\Lambda_{0})}}_{(r^{(3k-1)}_{2})}\otimes   
  \cdots \otimes {W_{L(\Lambda_{0})}}_{(r^{(2)}_{2})}\otimes {W_{L(\Lambda_{0})}}_{(r^{(1)}_{2})},
  \end{equation}
  where ${W_{L(\Lambda_{0})}}_{(r^{(t)}_{2})}$, $1 \leq t \leq 3k$ is a $\mathfrak{h}$-weighted subspace of $W_{L^A(3\Lambda_{0})}$ of weight $r^{(t)}_{2}\alpha_2$. From the condition (\ref{eq:7}), follows that monomial vectors
 \begin{equation}\label{eq:33}%\label{LN8}
\pi'_{\mathfrak{R^{-}}}\left(\pi_{\mathfrak{R^{-}}}b_a(\alpha_2)v_{L(k\Lambda_{0})}\right) 
 \end{equation} are elements of vector space (\ref{eq:32}). Now, using Georgiev's argument from \cite{G} follows $c_{a}=0$. 
\end{proof}

\section{Characters of principal subspaces}
\label{sec:9}
From Theorem~\ref{thm:1} we easily obtain the character of the principal subspace $W_{L(k\Lambda_{0})}$, 
\begin{align}\label{eq:34} 
 \text{ch} \ W_{L(k\Lambda_{0})}:=\sum_{m,r_1, r_2\geq 0} 
\text{dim} \ {W_{L(k\Lambda_{0})}}_{(m,r_1, r_2)}q^{m}y^{r_1}_{1}y^{r_2}_{2},
\end{align}
where ${W_{L(k\Lambda_{0})}}_{(m,r_1, r_2)}$ is a weight subspace spanned by monomial vectors of weight $-m$ and color-type $(r_1, r_2)$ (see \cite{Bu1}--\cite{Bu2}, \cite{G}).

If we write conditions on energies of quasi-particles of a basis $\mathfrak{B}_{W_{L(k\Lambda_0)}}$ in terms of the dual-charge-type (and the corresponding charge-type)
$$\left(r_{2}^{(1)},r_{2}^{(2)},\ldots ,r_{2}^{(3k)};  r_{1}^{(1)}, r_{1}^{(2)}, \ldots ,r_{1}^{(k)}\right): $$ 
\begin{align}\label{eq:35} 
\sum_{p=1}^{r^{(1)}_{2}}\sum_{q=1}^{r^{(1)}_1}\mathrm{min}\{n_{p,2},3n_{q,1}\}&=\sum_{s=1}^{k}r^{(s)}_1(r_{2}^{(2s)}+r_{2}^{(2s-1)}+r_{2}^{(2s-2)}),&\\
\label{eq:36} 
\sum_{p=1}^{r_{1}^{(1)}} (\sum_{p>p'>0}2\mathrm{min} \{ n_{p,1},
n_{p',1}\}+n_{p,1})&= \sum_{s=1}^{k}r^{(s)^{2}}_{1},\\
\label{eq:37} 
\sum_{p=1}^{r_{2}^{(1)}} (\sum_{p>p'>0}2\mathrm{min} \{ n_{p,2},
n_{p',2}\}+n_{p,2})&= \sum_{s=1}^{3k}r^{(s)^{2}}_{2},&
\end{align}
then, we have
\begin{theorem}
$$\mathrm{ch} \  W_{L(k\Lambda_{0})}
= \sum_{\substack{r^{(1)}_{1}\geq \ldots \geq r^{(k)}_{1}\geq 0\\ r^{(1)}_{2}\geq \ldots \geq r^{(3k)}_{2}\geq 0}}
\frac{q^{\sum_{s=1}^k r^{(s)^{2}}_{1}+\sum_{s=1}^{3k}r^{(s)^{2}}_{2}-\sum_{s=1}^k r^{(s)}_{1}(r^{(3s)}_{2}+r^{(3s-1)}_{2}+r^{(3s-2)}_{2})}}{(q)_{r^{(1)}_{1}-r^{(2)}_{1}}\cdots (q)_{r^{(k)}_{1}}(q)_{r^{(1)}_{2}-r^{(2)}_{2}}\cdots (q)_{r^{(3k)}_{2}}}y^{r_1}_{1}y^{r_2}_{2},
$$
where $(q)_0=1$, $(q)_r=(1-q)(1-q^2)\cdots (1-q^r)$ for $r > 0$, $r_1=\sum_{s=1}^{k}r_1^{(s)}$ and $r_2=\sum_{s=1}^{3k}r_2^{(s)}$.
$$\ \ \ \ \ \ \ \ \ \ \ \ \ \ \ \ \ \ \ \ \ \ \ \ \ \ \ \ \ \ \ \ \ \ \ \ \ \ \ \ \ \ \ \ \ \ \ \ \ \ \ \ \ \ \ \ \ \ \ \ \ \ \ \ \ \ \ \ \ \ \ \ \ \ \ \ \ \ \ \ \ \ \ \ \ \ \ \ \ \ \ \ \ \ \ \ \ \ \ \ \ \ \ \ \ \ \ \ \square$$
\end{theorem}
At the end we state the theorem in which we describe the basis of the principal subspace $W_{N(k\Lambda_{0})}$:
\begin{theorem}\label{thm:2} 
The set $\mathfrak{B}_{W_{N(k\Lambda_{0})}}=\left\{bv_{N(k\Lambda_{0})}:b \in B_{W_{N (k\Lambda_{0})}}\right\}$, where
$$B_{W_{N(k\Lambda_{0})}}= \bigcup_{\substack{n_{r_{1}^{(1)},1}\leq \ldots \leq n_{1,1}\\\substack{n_{r_{2}^{(1)},2}\leq \ldots \leq n_{1,2}}}}\left(\text{or, equivalently,} \ \ \ 
\bigcup_{\substack{r_{1}^{(1)}\geq  r_{1}^{(2)}\geq \cdots\geq 0\\\substack{ r_{2}^{(1)}\geq r_{2}^{(2)}\geq \cdots\geq 
0}}}\right)
$$
$$\left\{b\right.= b(\alpha_{2}) b(\alpha_{1})
=x_{n_{r_{2}^{(1)},2}\alpha_{2}}(m_{r_{2}^{(1)},2})\cdots x_{n_{1,2}\alpha_{2}}(m_{1,2}) x_{n_{r_{1}^{(1)},1}\alpha_{1}}(m_{r_{1}^{(1)},1})\cdots  x_{n_{1,1}\alpha_{1}}(m_{1,1}):$$
$$
\left|
\begin{array}{l}
m_{p,1}\leq  -n_{p,1}- \sum_{p>p'>0} 2 \ \text{min}\{n_{p,1}, n_{p',1}\},  \ 1\leq  p\leq r_{1}^{(1)};\\
m_{p+1,1} \leq   m_{p,1}-2n_{p,1} \  \text{if} \ n_{p+1,1}=n_{p,1}, \ 1\leq  p\leq r_{1}^{(1)}-1;\\
m_{p,2}\leq  -n_{p,2} + \sum_{q=1}^{r_{1}^{(1)}}\text{min}\left\{ 3n_{q,1},n_{p,2 }\right\} - \sum_{p>p'>0} 2 \ \text{min}\{n_{p,2}, n_{p',2}\}, \  1\leq  p\leq r_{2}^{(1)};\\
m_{p+1,2}\leq  m_{p,2}-2n_{p,2} \  \text{if} \ n_{p,2}=n_{p+1,2}, \  1\leq  p\leq r_{2}^{(1)}-1  
\end{array}\right\}
$$
is a basis of the principal subspace $W_{N(k\Lambda_{0})}$.
$$\ \ \ \ \ \ \ \ \ \ \ \ \ \ \ \ \ \ \ \ \ \ \ \ \ \ \ \ \ \ \ \ \ \ \ \ \ \ \ \ \ \ \ \ \ \ \ \ \ \ \ \ \ \ \ \ \ \ \ \ \ \ \ \ \ \ \ \ \ \ \ \ \ \ \ \ \ \ \ \ \ \ \ \ \ \ \ \ \ \ \ \ \ \ \ \ \ \ \ \ \ \ \ \ \ \ \ \ \square$$
\end{theorem}
The proof of the Theorem \ref{thm:2} is similar as in the case of $W_{L(k\Lambda_{0})}$ (see \cite{Bu1}), from which we can as before obtain the character of $W_{N(k\Lambda_{0})}$:
\begin{theorem}
$$\mathrm{ch} \  W_{N(k\Lambda_{0})}
= \sum_{\substack{r^{(1)}_{1}\geq \ldots \geq r^{(u)}_{1}\geq  0\\ r^{(1)}_{2}\geq \ldots \geq r^{(3v)}_{2}\geq  0\\ u,v \geq 0}}
\frac{q^{\sum_{s=1}^u r^{(s)^{2}}_{1}+\sum_{s=1}^{3v}r^{(s)^{2}}_{2}-\sum_{s\geq 1} r^{(s)}_{1}(r^{(3s)}_{2}+r^{(3s-1)}_{2}+r^{(3s-2)}_{2})}}{(q)_{r^{(1)}_{1}-r^{(2)}_{1}}\cdots (q)_{r^{(u)}_{1}}(q)_{r^{(1)}_{2}-r^{(2)}_{2}}\cdots (q)_{r^{(3v)}_{2}}}y^{r_1}_{1}y^{r_2}_{2},
$$
where $r_1=\sum_{s=1}^{u}r_1^{(s)}$ and $r_2=\sum_{s=1}^{3v}r_2^{(s)}$.
$$\ \ \ \ \ \ \ \ \ \ \ \ \ \ \ \ \ \ \ \ \ \ \ \ \ \ \ \ \ \ \ \ \ \ \ \ \ \ \ \ \ \ \ \ \ \ \ \ \ \ \ \ \ \ \ \ \ \ \ \ \ \ \ \ \ \ \ \ \ \ \ \ \ \ \ \ \ \ \ \ \ \ \ \ \ \ \ \ \ \ \ \ \ \ \ \ \ \ \ \ \ \ \ \ \ \ \ \ \square$$
\end{theorem}
From (\ref{eq:5}) and previous theorem follows a generalization of Euler-Cauchy theorem (cf. (2.2.8) and (2.2.9) in \cite{A1} and (4.1) in \cite{A2}):
\begin{theorem}\label{t5}
$$\prod_{m > 0} \frac{1}{(1-q^my_1)}\frac{1}{(1-q^my_2)}\frac{1}{(1-q^my_1y_2)}\frac{1}{(1-q^my_1y_2^2)}\frac{1}{(1-q^my_1y_2^3)}\frac{1}{(1-q^my_1^2y_2^3)}$$
\begin{equation}\label{eq:38}= \sum_{\substack{r^{(1)}_{1}\geq r^{(2)}_{1}\geq r^{(3)}_{1} \geq \ldots  \geq  0\\ r^{(1)}_{2}\geq r^{(2)}_{2} \geq r^{(3)}_{2}\geq \ldots  \geq 0}}
\frac{q^{\sum_{s\geq 1}  r^{(s)^{2}}_{1}+\sum_{s\geq 1} r^{(s)^{2}}_{2}-\sum_{s\geq 1} r^{(s)}_{1}(r^{(3s)}_{2}+r^{(3s-1)}_{2}+r^{(3s-2)}_{2})}}{(q)_{r^{(1)}_{1}-r^{(2)}_{1}}(q)_{r^{(2)}_{1}-r^{(3)}_{1}}\cdots  (q)_{r^{(1)}_{2}-r^{(2)}_{2}}(q)_{r^{(2)}_{2}-r^{(3)}_{2}}\cdots }y^{r_1}_{1}y^{r_2}_{2},
\end{equation}
where $r_1=\sum_{s\geq 1} r_1^{(s)}$ and $r_2=\sum_{s\geq 1} r_2^{(s)}$.
The sum on the right side of (\ref{eq:38}) is over all descending infinite sequences of non-negative integers with finite support.
$$\ \ \ \ \ \ \ \ \ \ \ \ \ \ \ \ \ \ \ \ \ \ \ \ \ \ \ \ \ \ \ \ \ \ \ \ \ \ \ \ \ \ \ \ \ \ \ \ \ \ \ \ \ \ \ \ \ \ \ \ \ \ \ \ \ \ \ \ \ \ \ \ \ \ \ \ \ \ \ \ \ \ \ \ \ \ \ \ \ \ \ \ \ \ \ \ \ \ \ \ \ \ \ \ \ \ \ \ \square$$
\end{theorem}

\section*{Acknowledgement}
I am very grateful to Mirko Primc for his help and valuable suggestions during the preparation of this work. I would like to thank to Dra\v zen Adamovi\' c for supporting my research. I also thank the referee for pointing out that Theorem \ref{t5} (Theorem \ref{tm1} in the Introduction) is not of Rogers-Ramanujan type, but rather a generalization of Euler-Cauchy identity.

\end{document}